\documentclass[11pt]{article} 
\usepackage{amsfonts,amsmath,latexsym,amssymb,mathrsfs,amsthm,comment}
\usepackage{graphicx}
\usepackage{xcolor}
\usepackage{booktabs}

\makeatletter
\@namedef{subjclassname@2020}{%
  \textup{2020} Mathematics Subject Classification}
\makeatother

\let\OLDthebibliography\thebibliography
\renewcommand\thebibliography[1]{
  \OLDthebibliography{#1}
  \setlength{\parskip}{1pt}
  \setlength{\itemsep}{0pt plus 0.0ex}
}

\def\numberlikeadb{\global\def\theequation{\thesection.\arabic{equation}}}
\numberlikeadb
\newtheorem{theorem}{Theorem}[section]

\newtheorem{corollary}[theorem]{Corollary}

\newtheorem{remark}[theorem]{Remark}

\usepackage[multiple]{footmisc}

\usepackage{lscape}
\usepackage{caption}
\usepackage{multirow}
\allowdisplaybreaks

\begin{document} 


\title{Infinite Divisibility of the Product of Two Correlated Normal Random Variables and Exact Distribution of the Sample Mean\footnote{{\bf{This article is dedicated to Professor Mourad Ismail on the occasion of his 80th birthday.}}}}

\author{Robert E. Gaunt\footnote{Department of Mathematics, The University of Manchester, Oxford Road, Manchester M13 9PL, UK, 
robert.gaunt@manchester.ac.uk; Saralees.Nadarajah@manchester.ac.uk},\: Saralees Nadarajah$^\dagger$\ and Tibor K. Pog\'any\footnote{Institute of Applied Mathematics, \'Obuda University, 1034 Budapest, Hungary, pogany.tibor@nik.uni-obuda.hu}  
\footnote{Faculty of Maritime Studies, University of Rijeka, 51000 Rijeka, Croatia, tibor.poganj@uniri.hr}}

\date{} 
\maketitle

\vspace{-5mm}

\begin{abstract} We prove that the distribution of the product of two correlated normal random variables with arbitrary means and arbitrary variances is infinitely divisible. We also obtain exact formulas for the probability density function of the sum of independent copies of such random variables.
\end{abstract}

\noindent{{\bf{Keywords:}}} Product of correlated normal random variables; infinite divisibility; sum of independent random variables; probability density function; confluent hypergeometric function; modified Bessel function

\noindent{{{\bf{AMS 2020 Subject Classification:}}} Primary 60E05; 60E07; 62E15; Secondary 33C15; 60E10
}


\section{Introduction}

Let $(X, Y)$ be a bivariate normal random vector with mean vector $(\mu_X,\mu_Y)\in\mathbb{R}^2$, variances $(\sigma_X^2,\sigma_Y^2)\in(0,\infty)^2$ 
and correlation coefficient $\rho\in(-1,1)$. Since the work of \cite{craig,wb32} in the 1930's, the distribution of the product 
$Z=XY$
has received much attention in the statistics 
literature (see \cite{gaunt22,np16} for an overview of some of the literature), and has found numerous applications, with 
recent examples including chemical physics \cite{hey}, condensed matter physics \cite{ach} and astrophysics \cite{cac}. 
The mean $\overline{Z}_n=n^{-1}\sum_{i=1}^nZ_i$, where $Z_1,\ldots,Z_n$ are independent copies of the product $Z$, has also found 
applications in fields such as quantum cosmology \cite{gr}, astrophysics \cite{man} and electrical engineering \cite{ware}.



Recently, \cite{cui} obtained the following exact formula for the probability density function (PDF) of the product $Z$: 
   \begin{align} \label{pdf}
	    f_{Z}(x) &= \frac{1}{\pi}\exp\bigg\{-\frac{1}{2(1-\rho^2)}\bigg(\frac{\mu_X^2}{\sigma_X^2}+\frac{\mu_Y^2}{\sigma_Y^2} 
			          - \frac{2\rho(x+\mu_X\mu_Y)}{\sigma_X\sigma_Y}\bigg)\bigg\} \notag \\
               &\quad \times \sum_{k=0}^\infty\sum_{j=0}^{2k}\frac{x^{2k-j}|x|^{j-k}\sigma_X^{j-k-1}}{(2k)!(1-\rho^2)^{2k+1/2}
							    \sigma_Y^{j-k+1}}\binom{2k}{j}\bigg(\frac{\mu_X}{\sigma_X^2}-\frac{\rho \mu_Y}{\sigma_X\sigma_Y}\bigg)^j \notag \\
               &\quad \times \bigg(\frac{\mu_Y}{\sigma_Y^2}-\frac{\rho \mu_X}{\sigma_X\sigma_Y}\bigg)^{2k-j}K_{j-k} 
							    \bigg(\frac{|x|}{\sigma_X\sigma_Y(1-\rho^2)}\bigg), \quad x\in\mathbb{R},
   \end{align}
where $K_\nu(x)$  is a modified Bessel function of the second 
kind.
In Appendix \ref{appa}, we define the modified Bessel function of the second kind and the other special functions used in this paper. We also state some relevant basic properties of these special functions that are required in this paper; we will indeed draw heavily on special function theory in this paper.

If one of the means, say $\mu_Y$, is equal to zero and $\rho=0$, then the PDF \eqref{pdf} reduces to a single infinite 
series (see \cite{cui,simon}): 
   \begin{equation} \label{simple}
      f_Z(x) = \frac1{\pi\sigma_X\sigma_Y}\exp\bigg(\!-\frac{\mu_X^2}{2\sigma_X^2}\bigg)\sum_{j=0}^\infty 
			         \frac{\mu_X^{2j}|x|^j}{(2j)!\sigma_X^{3j}\sigma_Y^j} K_j\bigg(\frac{|x|}{\sigma_X\sigma_Y}\bigg), 
							 \quad x\in\mathbb{R}.   
   \end{equation} 
We further observe that if $\mu_X/\sigma_X=\rho\mu_Y/\sigma_Y$ or $\mu_Y/\sigma_Y=\rho\mu_X/\sigma_X$, then the formula (\ref{pdf}) simplifies to a single infinite series. For example, if $\mu_X/\sigma_X=\rho\mu_Y/\sigma_Y$, then, for $x\in\mathbb{R}$,
  \begin{align} 
	    f_{Z}(x) &= \frac1{\pi\sigma_X\sigma_Y\sqrt{1-\rho^2}}\exp\bigg\{-\frac{\mu_Y^2}{2\sigma_Y^2}+\frac{\rho x}{\sigma_X\sigma_Y(1-\rho^2)}\bigg\} \notag \\
               &\quad \times \sum_{j=0}^\infty\frac{1}{(2j)!
							    }\bigg(\frac{\mu_Y}{\sigma_Y}\bigg)^{2j}\bigg(\frac{|x|}{\sigma_X\sigma_Y}\bigg)^j K_{j} 
							    \bigg(\frac{|x|}{\sigma_X\sigma_Y(1-\rho^2)}\bigg).\label{red1}
   \end{align}  
In the case $\mu_X=\mu_Y=0$, the PDF (\ref{pdf}) simplifies further to
\begin{equation}\label{single}f_{Z}(x)=\frac{1}{\pi\sigma_X\sigma_Y\sqrt{1-\rho^2}}\exp\bigg(\frac{\rho  x}{\sigma_X\sigma_Y(1-\rho^2)} \bigg)K_{0}\bigg(\frac{ |x|}{\sigma_X\sigma_Y(1-\rho^2)}\bigg), \quad x\in\mathbb{R},
\end{equation}   
and formula (\ref{single}) 
generalises in a natural manner to yield the following 
exact formula for the PDF of $\overline{Z}_n=n^{-1}\sum_{i=1}^nZ_i$, which was derived independently by \cite{man,np16,wb32}:
\begin{equation}\label{bar}f_{\overline{Z}_n}(x)=\frac{2^{(1-n)/2}|x|^{(n-1)/2}}{s_n^{(n+1)/2}\sqrt{\pi(1-\rho^2)}\Gamma(n/2)}\exp\bigg(\frac{\rho  x}{s_n(1-\rho^2)} \bigg)K_{\frac{n-1}{2}}\bigg(\frac{ |x|}{s_n(1-\rho^2)}\bigg), \quad x\in\mathbb{R},
\end{equation}
where $s_n=\sigma_X\sigma_Y/n$, and \cite{gauntprod} observed that $\overline{Z}_n$ is variance-gamma distributed. 
However, until now an exact formula for the PDF of the mean $\overline{Z}_n$ was not known for the case of a general mean vector $(\mu_X,\mu_Y)\in\mathbb{R}^2$. 

In this paper, we obtain exact formulas for the PDF of the sum $S_n=\sum_{i=1}^n Z_i$
for the full range of parameter values; exact formulas for the PDF of the sample mean $\overline{Z}_n$ follow from a simple rescaling, as noted in Remark \ref{remscale}. We provide an infinite series representation of the PDF (Theorem \ref{thm1}) and an integral representation of the PDF (Theorem \ref{thmint}). As is the case for the PDF of the product $Z$, our series representation of the PDF of the sum $S_n$ simplifies for certain parameter values, and we note these simplifications in Remark \ref{remsimple} and Theorem \ref{rho0}. The integral representation of the PDF is rather useful and has recently been applied by \cite{gz24} to derive the entire asymptotic expansion for the PDF and tail probabilities of the sum $S_n$ for the full range of parameter values.

We also prove that the distribution of the product $Z$ is infinitely divisible for all parameter values (Theorem \ref{thmif}). Our proof is constructive and relies on the analysis we used to prove Theorem \ref{thm1}. 
Over the years, there has been much interest in determining whether given probability distributions are infinitely divisible.
Infinite divisibility has been established for a number of classical distributions, such as the normal, gamma and stable distributions (immediate from their characteristic functions), Student's $t$ distribution \cite{g76,i77}, the $F$-distribution \cite{ik79}, the lognormal distribution \cite{t77}, the hyperbolic and generalized inverse Gaussian distributions \cite{bh77} and the von-Mises Fisher distribution \cite{kent}, as well as more exotic distributions involving special functions such as \cite{ba23,bkp21}; many further examples, as well as theory and applications can be found in the book \cite{s03}. It was in fact already known (see \cite[Section 2.4]{gaunt22}) that the distribution of the product $Z$ is infinitely divisible in the case $\mu_X=\mu_Y=0$. This is because $Z$ is variance-gamma distributed in this case (see \cite{gaunt vg,gauntprod}), and the variance-gamma distribution is infinitely divisible (see \cite{vg review}). Our contribution is to extend this result to the full range of parameter values.  




\section{Results and proofs}

\subsection{Infinite series representations of the density}

We begin by stating an exact formula for the PDF of the sum $S_n=\sum_{i=1}^nZ_i$, which is valid for the full range of parameter values. The formula is expressed in terms of the confluent hypergeometric function of the second kind, which is defined in Appendix \ref{appa}.
We denote the sign function by $\mathrm{sgn}(x)$, which is given by $\mathrm{sgn}(x)=1$ for $x>0$, $\mathrm{sgn}(0)=0$, $\mathrm{sgn}(x)=-1$ for $x<0$. Finally, we let $a_{j,k}(x)=k-j$ if $x\geq0$ and $a_{j,k}(x)=j$ if $x<0$.

\begin{theorem}\label{thm1}
For $x\in\mathbb{R}$,
\begin{align}
f_{S_n}(x)&=\frac{(1-\rho^2)^{n/2-1}}{2^{n-1}\sigma_X\sigma_Y}\exp\bigg(-\frac{n}{2(1-\rho^2)}\bigg(\frac{\mu_X^2}{\sigma_X^2}+\frac{\mu_Y^2}{\sigma_Y^2}-\frac{2\rho\mu_X\mu_Y}{\sigma_X\sigma_Y}\bigg)-\frac{|x|}{\sigma_X\sigma_Y(1+\rho\,\mathrm{sgn}(x))}\bigg)\nonumber\\
&\quad\times\sum_{k=0}^\infty\sum_{j=0}^k\frac{(n/8)^k}{k!\Gamma(n/2+a_{j,k}(x))}\binom{k}{j} \bigg(\frac{1+\rho}{1-\rho}\bigg)^j\bigg(\frac{\mu_X}{\sigma_X}-\frac{\mu_Y}{\sigma_Y}\bigg)^{2j} \bigg(\frac{1-\rho}{1+\rho}\bigg)^{k-j}\nonumber\\
&\quad\times \bigg(\frac{\mu_X}{\sigma_X}+\frac{\mu_Y}{\sigma_Y}\bigg)^{2k-2j} U\bigg(1-\frac{n}{2}-a_{j,k}(x),2-n-k,\frac{2|x|}{\sigma_X\sigma_Y(1-\rho^2)}\bigg). \label{for1} 
\end{align}   
\end{theorem}

Before proving Theorem \ref{thm1}, we make some observations and remarks and note some simplifications of the PDF (\ref{thm1}) for certain parameter values.

\begin{corollary}
The PDF $f_{S_n}(x)$ is bounded for all $x\in\mathbb{R}$ if and only if $n\geq2$.    
\end{corollary}

\begin{proof}
It was shown in \cite{gz23} that the PDF of $Z=_d S_1$ (here $=_d$ denotes equality in distribution) is unimodal with a logaritmic singularity at the mode $x=0$ (this can also be inferred from formula (\ref{for1}) by using the limiting forms (\ref{00}) and (\ref{000})). For $n\geq2$, the PDF (\ref{for1}) of $S_n$ is, however, bounded for all $x\in\mathbb{R}$, as can be seen from applying the limiting form to (\ref{000}) to determine the behaviour as $x\rightarrow0$, and due to the fact that the function $U(a,b,x)$ only has (possible) singularities in the limits $x\rightarrow0$ and $x\rightarrow\pm\infty$, and the PDF is clearly bounded in the limits $x\rightarrow\pm\infty$ (and the asymptotic behaviour of the PDF in the limits $x\rightarrow\pm\infty$ is in fact given in Theorem 3.1 of \cite{gz24}, which confirms that $f_{S_n}(x)\rightarrow0$ as $x\rightarrow\pm\infty$).
\end{proof}

\begin{remark}\label{remscale} Since $\overline{Z}_n=S_n/n$, we have by a simple rescaling that $f_{\overline{Z}_n}(x)=nf_{S_n}(nx)$. Applying this formula to (\ref{for1}) yields the following formula for the PDF of the sample mean $\overline{Z}_n$:
\begin{align}
f_{\overline{Z}_n}(x)&=\frac{n(1-\rho^2)^{n/2-1}}{2^{n-1}\sigma_X\sigma_Y}\exp\bigg(-\frac{n}{2(1-\rho^2)}\bigg(\frac{\mu_X^2}{\sigma_X^2}+\frac{\mu_Y^2}{\sigma_Y^2}-\frac{2\rho\mu_X\mu_Y}{\sigma_X\sigma_Y}\bigg)-\frac{n|x|}{\sigma_X\sigma_Y(1+\rho\,\mathrm{sgn}(x))}\bigg)\nonumber\\
&\quad\times\sum_{k=0}^\infty\sum_{j=0}^k\frac{(n/8)^k}{k!\Gamma(n/2+a_{j,k}(x))}\binom{k}{j} \bigg(\frac{1+\rho}{1-\rho}\bigg)^j\bigg(\frac{\mu_X}{\sigma_X}-\frac{\mu_Y}{\sigma_Y}\bigg)^{2j} \bigg(\frac{1-\rho}{1+\rho}\bigg)^{k-j}\nonumber\\
&\quad\times \bigg(\frac{\mu_X}{\sigma_X}+\frac{\mu_Y}{\sigma_Y}\bigg)^{2k-2j} U\bigg(1-\frac{n}{2}-a_{j,k}(x),2-n-k,\frac{2n|x|}{\sigma_X\sigma_Y(1-\rho^2)}\bigg),\quad x\in\mathbb{R}. \nonumber
\end{align} 
We can similarly deduce formulas for the PDF of the sample mean $\overline{Z}_n$ from all our subsequent formulas for the PDF of the sum $S_n$.
\end{remark}

\begin{remark}Setting $n=1$ in (\ref{for1}) of course yields an exact formula for the PDF of the product $Z$. The resulting formula takes on a different form to the formula (\ref{pdf}) of \cite{cui}, expressed in terms of a double infinite series involving the confluent hypergeometric function of the second kind rather than a double infinite series involving the modified Bessel function of the second kind. This difference arises because the derivations are different.  Formula (\ref{pdf}) of \cite{cui} was obtained from a calculation using the standard general formula for the PDF of the product $Z$ of two random variables $X$ and $Y$, as given by $f_Z(x)=\int_{-\infty}^\infty f_{X,Y}(t,x/t) |t|^{-1}\,\mathrm{d}t$, where in their case $f_{X,Y}$ is the PDF of the bivariate normal random vector $(X,Y)$. On the other hand, we use a characteristic function argument, which is well-suited to sums of independent random variables.

Both formulas (\ref{pdf}) and (\ref{for1}) (when specialised to $n=1$) are of a similar level of complexity. One notable difference is that, in contrast to formula (\ref{pdf}) of \cite{cui}, the summands in the infinite series of formula (\ref{for1}) are always non-negative for all values of $x$ and all parameter values (due to the positivity of the confluent hypergeometric function of the second kind; see (\ref{positive})). Another way in which the formulas differ is that they reduce to a single infinite series for different parameter values, as shown in the next remark.   
\end{remark}

\begin{remark}\label{remsimple} 1. The double infinite series (\ref{for1}) reduces to a single infinite series for the following parameter values.

Suppose that $\mu_X/\sigma_X=\mu_Y/\sigma_Y$. Then, for $x\in\mathbb{R}$,
\begin{align*}
f_{S_n}(x)&=\frac{(1-\rho^2)^{n/2-1}}{2^{n-1}\sigma_X\sigma_Y}\exp\bigg(-\frac{n\mu_X^2}{\sigma_X^2(1+\rho)}-\frac{|x|}{\sigma_X\sigma_Y(1+\rho\,\mathrm{sgn}(x))}\bigg)\sum_{k=0}^\infty\frac{(n/2)^k}{k!\Gamma(n/2+a_{0,k}(x))}\\
&\quad \times\bigg(\frac{1-\rho}{1+\rho}\bigg)^{k} \bigg(\frac{\mu_X}{\sigma_X}\bigg)^{2k} U\bigg(1-\frac{n}{2}-a_{0,k}(x),2-n-k,\frac{2|x|}{\sigma_X\sigma_Y(1-\rho^2)}\bigg). 
\end{align*}
Now, suppose that $\mu_X/\sigma_X=-\mu_Y/\sigma_Y$. Then, for $x\in\mathbb{R}$,
\begin{align*}
f_{S_n}(x)&=\frac{(1-\rho^2)^{n/2-1}}{2^{n-1}\sigma_X\sigma_Y}\exp\bigg(-\frac{n\mu_X^2}{\sigma_X^2(1-\rho)}-\frac{|x|}{\sigma_X\sigma_Y(1+\rho\,\mathrm{sgn}(x))}\bigg)\sum_{k=0}^\infty\frac{(n/2)^k}{k!\Gamma(n/2+a_{k,k}(x))} \\
&\quad\times\bigg(\frac{1+\rho}{1-\rho}\bigg)^{k} \bigg(\frac{\mu_X}{\sigma_X}\bigg)^{2k} U\bigg(1-\frac{n}{2}-a_{k,k}(x),2-n-k,\frac{2|x|}{\sigma_X\sigma_Y(1-\rho^2)}\bigg). 
\end{align*}
We therefore see that formula (\ref{for1}) reduces to a single infinite series in the cases $\mu_X/\sigma_X=\mu_Y/\sigma_Y$ and $\mu_X/\sigma_X=-\mu_Y/\sigma_Y$. This is in contrast to the formula (\ref{pdf}) of \cite{cui} that reduces to a single infinite series if $\mu_X/\sigma_X=\rho\mu_Y/\sigma_Y$ or $\mu_Y/\sigma_Y=\rho\mu_X/\sigma_X$ (see (\ref{red1})).

\vspace{2mm}

\noindent 2. Just like formula (\ref{pdf}), the double infinite series (\ref{for1}) reduces to a single term in the case $\mu_X=\mu_Y=0$: 
\begin{align*}
f_{S_n}(x)=\frac{(1-\rho^2)^{n/2-1}}{2^{n-1}s\Gamma(n/2)}\exp\bigg(-\frac{|x|}{s(1+\rho\,\mathrm{sgn}(x))}\bigg)U\bigg(1-\frac{n}{2},2-n,\frac{2|x|}{s(1-\rho^2)}\bigg), \quad x\in\mathbb{R},
\end{align*}
where $s=\sigma_X\sigma_Y$. On applying the relation (\ref{uk}) followed by the identity (\ref{par}) we see that this formula can be expressed in terms of the modified Bessel function of the second kind: 
\begin{align*}
f_{S_n}(x)=\frac{(n/2)^{(n-1)/2}|x|^{(n-1)/2}}{s^{(n+1)/2}\sqrt{\pi(1-\rho^2)}\Gamma(n/2)}\exp\bigg(\frac{\rho  x}{s(1-\rho^2)} \bigg)K_{\frac{n-1}{2}}\bigg(\frac{ |x|}{s(1-\rho^2)}\bigg), \quad x\in\mathbb{R},  
\end{align*}
which, on recalling that  $f_{\overline{Z}_n}(x)=nf_{S_n}(nx)$, can be seen to be in agreement with formula (\ref{bar}) for the PDF of $\overline{Z}_n$ in the case $\mu_X=\mu_Y=0$.
\end{remark}

The following theorem gives an exact formula for the PDF of the sum $S_n$ in the special case that $\rho=0$ and one of the means is equal to zero; without loss of generality we set $\mu_Y=0$. Since the formula is expressed as a single infinite series involving the modified Bessel function of the second kind, it cannot be easily recovered from our general formula (\ref{for1}), so we provide a separate proof.

\begin{theorem}\label{rho0}Suppose $\rho=0$, $\mu_Y=0$, and $\mu_X\in\mathbb{R}$.
 Then, for $x\in\mathbb{R}$,
\begin{align}\label{simple2}
f_{S_n}(x)=\frac{2^{(1-n)/2}}{(\sigma_X\sigma_Y)^{(n+1)/2}\sqrt{\pi}}\exp\bigg(-\frac{n\mu_X^2}{2\sigma_X^2}\bigg)\sum_{k=0}^{\infty}\frac{(n\mu_X^2/4)^k|x|^{(n-1)/2+k}}{k!\Gamma(n/2+k)\sigma_X^{3k}\sigma_Y^k}K_{\frac{n-1}{2}+k}\bigg(\frac{|x|}{\sigma_X\sigma_Y}\bigg).    
\end{align}
\end{theorem}

\begin{remark} Theorem \ref{rho0} extends in a natural way the PDF $f_Z(x)$ described in \eqref{simple} to the PDF $f_{S_n}(x)$. On calculating $(2j)!=\Gamma(2j+1)$ using the gamma duplication formula $\Gamma(2x)=\pi^{-1/2}2^{2x-1}\Gamma(x)\Gamma(x
+1/2)$ see {\rm (\cite[Section 5.5(iii)]{olver})} we see that formula \eqref{simple2} with $n=1$ indeed reduces to \eqref{simple}.
\end{remark}

We now set about proving Theorems \ref{thm1} and \ref{rho0}. For ease of notation, we shall prove our results for the case $\sigma_X=\sigma_Y=1$, with the results for the general case $\sigma_X,\sigma_Y>0$ following from the basic relation that $Z=XY=_d \sigma_X\sigma_Y UV$, where $(U,V)$ is a bivariate normal random vector with mean vector $(\mu_X/\sigma_X,\mu_Y/\sigma_Y)$, variances $(1,1)$ and correlation coefficient $\rho$.

Our proofs will employ characteristic function arguments. We will make use of the characteristic function of $Z$, which we denote by $\varphi_Z(t)=\mathbb{E}[\mathrm{e}^{\mathrm{i}tZ}]=\int_{-\infty}^\infty\mathrm{e}^{\mathrm{i}tx} f_Z(x)\,\mathrm{d}x$. By \cite[equation (10)]{craig}, we have, for $t\in\mathbb{R}$, 
 \begin{align}\label{charz}
      \varphi_Z(t) = \frac1{([1-(1+\rho)\mathrm{i}t][1+(1-\rho)\mathrm{i}t])^{1/2}} 
			             \exp\bigg(\frac{-(\mu_X^2+\mu_Y^2-2\rho\mu_X\mu_Y)t^2+2\mu_X\mu_Y\mathrm{i}t}
									 {2[1-(1+\rho)\mathrm{i}t][1+(1-\rho)\mathrm{i}t]}\bigg).
   \end{align}
We are now in a position to prove Theorem \ref{thm1}.   

\vspace{2mm}   

\noindent{\emph{Proof of Theorem \ref{thm1}.}} 
Let $\varphi_{S_n}(t)=\mathbb{E}[\mathrm{e}^{\mathrm{i}tS_n}]$ denote the characteristic function of $S_n$. 
Then
   \begin{align*}
      \varphi_{S_n}(t) = \frac1{([1-(1+\rho)\mathrm{i}t][1+(1-\rho)\mathrm{i}t])^{n/2}} 
			             \exp\bigg(\frac{-n(\mu_X^2+\mu_Y^2-2\rho\mu_X\mu_Y)t^2+2n\mu_X\mu_Y\mathrm{i}t}
									 {2[1-(1+\rho)\mathrm{i}t][1+(1-\rho)\mathrm{i}t]}\bigg),
   \end{align*}
which follows from (\ref{charz}) and the standard formula for the characteristic function of sums 
of independent random variables.
We note the partial fraction decomposition
\begin{align*}
&\frac{-(\mu_X^2+\mu_Y^2-2\rho\mu_X\mu_Y)t^2+2\mu_X\mu_Y\mathrm{i}t}{2[1-(1+\rho)\mathrm{i}t][1+(1-\rho)\mathrm{i}t]}\\
&=-\frac{(\mu_X^2+\mu_Y^2-2\rho\mu_X\mu_Y)}{2(1-\rho^2)}+\frac{1}{4(1-\rho^2)}\bigg(\frac{(1+\rho)(\mu_X-\mu_Y)^2}{(1+(1-\rho)\mathrm{i}t)}+\frac{(1-\rho)(\mu_X+\mu_Y)^2}{(1-(1+\rho)\mathrm{i}t)}\bigg).
\end{align*}
Therefore we can write
\begin{align*}
\varphi_{S_n}(t)&=\frac{1}{([1-(1+\rho)\mathrm{i}t][1+(1-\rho)\mathrm{i}t])^{n/2}}\exp\bigg(-\frac{n(\mu_X^2+\mu_Y^2-2\rho\mu_X\mu_Y)}{2(1-\rho^2)}\bigg)\\
&\quad\times\exp\bigg(\frac{n}{4(1-\rho^2)}\bigg(\frac{(1+\rho)(\mu_X-\mu_Y)^2}{(1+(1-\rho)\mathrm{i}t)}+\frac{(1-\rho)(\mu_X+\mu_Y)^2}{(1-(1+\rho)\mathrm{i}t)}\bigg)\bigg)\\
&=\frac{1}{([1-(1+\rho)\mathrm{i}t][1+(1-\rho)\mathrm{i}t])^{n/2}}\exp\bigg(-\frac{n(\mu_X^2+\mu_Y^2-2\rho\mu_X\mu_Y)}{2(1-\rho^2)}\bigg)\\
&\quad\times\sum_{k=0}^\infty\frac{(n/4)^k}{k!(1-\rho^2)^k}\bigg(\frac{(1+\rho)(\mu_X-\mu_Y)^2}{(1+(1-\rho)\mathrm{i}t)}+\frac{(1-\rho)(\mu_X+\mu_Y)^2}{(1-(1+\rho)\mathrm{i}t)}\bigg)^k\\
&=\exp\bigg(-\frac{n(\mu_X^2+\mu_Y^2-2\rho\mu_X\mu_Y)}{2(1-\rho^2)}\bigg)\sum_{k=0}^\infty\sum_{j=0}^k\frac{(n/4)^k}{k!(1-\rho^2)^k}\binom{k}{j}(1+\rho)^j(\mu_X-\mu_Y)^{2j}\\
&\quad\times (1-\rho)^{k-j}(\mu_X+\mu_Y)^{2k-2j}(1+(1-\rho)\mathrm{i}t)^{-n/2-j}(1-(1+\rho)\mathrm{i}t)^{-n/2-k+j}.
\end{align*}
Hence, by the inversion theorem, the PDF of $S_n$ can be expressed as
\begin{align*}
f_{S_n}(x)&=\frac{1}{2\pi} \exp\bigg(-\frac{n(\mu_X^2+\mu_Y^2-2\rho\mu_X\mu_Y)}{2(1-\rho^2)}\bigg)\sum_{k=0}^\infty\sum_{j=0}^k\frac{(n/4)^k}{k!(1-\rho^2)^k}\binom{k}{j}(1+\rho)^j(\mu_X-\mu_Y)^{2j}\\
&\quad\times (1-\rho)^{k-j}(\mu_X+\mu_Y)^{2k-2j}\int_{-\infty}^\infty\mathrm{e}^{-\mathrm{i}xt}(1+(1-\rho)\mathrm{i}t)^{-n/2-j}(1-(1+\rho)\mathrm{i}t)^{-n/2-k+j}\,\mathrm{d}t\\
&=\frac{1}{2\pi} \exp\bigg(-\frac{n(\mu_X^2+\mu_Y^2-2\rho\mu_X\mu_Y)}{2(1-\rho^2)}\bigg)\sum_{k=0}^\infty\sum_{j=0}^k\frac{(n/4)^k}{k!(1-\rho^2)^{n/2+k}}\binom{k}{j}\\
&\quad\times\bigg(\frac{1+\rho}{1-\rho}\bigg)^j(\mu_X-\mu_Y)^{2j}\bigg(\frac{1-\rho}{1+\rho}\bigg)^{k-j}(\mu_X+\mu_Y)^{2k-2j}\\
&\quad\times \int_{-\infty}^\infty\mathrm{e}^{-\mathrm{i}xt}\bigg(\frac{1}{1-\rho}+\mathrm{i}t\bigg)^{-n/2-j}\bigg(\frac{1}{1+\rho}-\mathrm{i}t\bigg)^{-n/2-k+j}\,\mathrm{d}t.
\end{align*}
Evaluating the integral using (\ref{int1}) gives that, for $x\geq0$,
\begin{align}
f_{S_n}(x)&=\frac{1}{2\pi} \exp\bigg(-\frac{n(\mu_X^2+\mu_Y^2-2\rho\mu_X\mu_Y)}{2(1-\rho^2)}\bigg)\sum_{k=0}^\infty\sum_{j=0}^k\frac{(n/4)^k}{k!(1-\rho^2)^{n/2+k}}\binom{k}{j}\bigg(\frac{1+\rho}{1-\rho}\bigg)^j\nonumber\\
&\quad\times (\mu_X-\mu_Y)^{2j}\bigg(\frac{1-\rho}{1+\rho}\bigg)^{k-j}(\mu_X+\mu_Y)^{2k-2j}\cdot \frac{2\pi}{\Gamma(n/2+k-j)}\bigg(\frac{2}{1-\rho^2}\bigg)^{1-n-k}\nonumber\\
&\quad\times \exp\bigg(-\frac{x}{1+\rho}\bigg) U\bigg(1-\frac{n}{2}-k+j,2-n-k,\frac{2x}{1-\rho^2}\bigg)\nonumber\\
&=\frac{(1-\rho^2)^{n/2-1}}{2^{n-1}}\exp\bigg(-\frac{n(\mu_X^2+\mu_Y^2-2\rho\mu_X\mu_Y)}{2(1-\rho^2)}-\frac{x}{1+\rho}\bigg)\nonumber\\
&\quad\times\sum_{k=0}^\infty\sum_{j=0}^k\frac{(n/8)^k}{k!\Gamma(n/2+k-j)}\binom{k}{j} \bigg(\frac{1+\rho}{1-\rho}\bigg)^j(\mu_X-\mu_Y)^{2j} \bigg(\frac{1-\rho}{1+\rho}\bigg)^{k-j}\nonumber\\
&\quad\times(\mu_X+\mu_Y)^{2k-2j} U\bigg(1-\frac{n}{2}-k+j,2-n-k,\frac{2x}{1-\rho^2}\bigg).\label{56}
\end{align}
Similarly, we can find a formula for $x<0$:
\begin{align}f_{S_n}(x)&=\frac{(1-\rho^2)^{n/2-1}}{2^{n-1}}\exp\bigg(-\frac{n(\mu_X^2+\mu_Y^2-2\rho\mu_X\mu_Y)}{2(1-\rho^2)}+\frac{x}{1-\rho}\bigg)\nonumber\\
&\quad\sum_{k=0}^\infty\sum_{j=0}^k\frac{(n/8)^k}{k!\Gamma(n/2+j)}\binom{k}{j} \bigg(\frac{1+\rho}{1-\rho}\bigg)^j(\mu_X-\mu_Y)^{2j} \bigg(\frac{1-\rho}{1+\rho}\bigg)^{k-j}\nonumber\\
&\quad\times(\mu_X+\mu_Y)^{2k-2j} U\bigg(1-\frac{n}{2}-j,2-n-k,\frac{2|x|}{1-\rho^2}\bigg).\label{78}
\end{align}
We now obtain the desired formula (\ref{for1}) from equations (\ref{56}) and (\ref{78}) in the general case $\sigma_X,\sigma_Y>0$ by a simple rescaling and then using the notation $\mathrm{sgn}(x)$ and $a_{j,k}(x)$ in order to write down a compact expression that holds for all $x\in\mathbb{R}$.
\qed

\vspace{2mm} 

\noindent \emph{Proof of Theorem \ref{rho0}.} In the case $\rho=\mu_Y=0$ and $\mu_X=\mu$, the characteristic function simplifies to
\begin{align*}
\varphi_{S_n}(t)&= \frac{1}{(1+t^2)^{n/2}}\exp\bigg(-\frac{n\mu^2t^2}{2(1+t^2)}\bigg) \\
&=\frac{1}{(1+t^2)^{n/2}}\exp\bigg(-\frac{n\mu^2}{2}\bigg)\exp\bigg(\frac{n\mu^2}{2(1+t^2)}\bigg)\\
&=\exp\bigg(-\frac{n\mu^2}{2}\bigg)\sum_{k=0}^\infty\frac{1}{k!}\bigg(\frac{n\mu^2}{2}\bigg)^k\frac{1}{(1+t^2)^{n/2+k}}.
\end{align*}
Hence, by the inversion theorem, the PDF of $S_n$ can be expressed as
\begin{align*}
f_{S_n}(x)&=\frac{1}{2\pi} \exp\bigg(-\frac{n\mu^2}{2}\bigg)\sum_{k=0}^\infty \frac{1}{k!}\bigg(\frac{n\mu^2}{2}\bigg)^k\int_{-\infty}^\infty  \frac{\mathrm{e}^{-\mathrm{i}xt}}{(1+t^2)^{n/2+k}}\,\mathrm{d}t \\
&=\frac{1}{2^{(n-1)/2}\sqrt{\pi}}\exp\bigg(-\frac{n\mu^2}{2}\bigg)\sum_{k=0}^{\infty}\frac{(n\mu^2/4)^k}{k!\Gamma(n/2+k)}|x|^{(n-1)/2+k}K_{\frac{n-1}{2}+k}(|x|),
\end{align*}
where we calculated the integral using (\ref{int11}).
\qed

\subsection{Integral representations of the density}

In the following theorem, we provide an alternative representation of the Neumann series of the first type built by confluent hypergeometric 
 functions of the second kind (see e.g.\ \cite{BJMP, GNW}) in \eqref{for1} by presenting the PDF by an integral representation. The formulas are expressed in terms of the modified Bessel function of the first kind, which is defined in Appendix \ref{appa}.

\begin{theorem}\label{thmint} 
1. Suppose $|\mu_X|/\sigma_X\not=|\mu_Y|/\sigma_Y$. Then, for $x>0$,
   \begin{align} \label{Watson2}
	    f_{S_n}(x) &= D_1(x) \int_0^\infty (t(t+1))^{(n-2)/4} \exp\bigg(-\frac{2x t}{\sigma_X\sigma_Y(1-\rho^2)}\bigg)  \notag \\ 
						 &\quad \times I_{\frac{n}2-1}\bigg(\bigg|\frac{\mu_X}{\sigma_X}-\frac{\mu_Y}{\sigma_Y}\bigg|\frac{\sqrt{nxt}}{(1-\rho)\sqrt{\sigma_X\sigma_Y}}\bigg) I_{\frac{n}2-1}\bigg(\bigg|\frac{\mu_X}{\sigma_X}+\frac{\mu_Y}{\sigma_Y}\bigg|\frac{\sqrt{nx(1+t)}}{(1+\rho)\sqrt{\sigma_X\sigma_Y}}\bigg)	\,{\rm d} t,
	 \end{align}	
where  
   \begin{align*}
	    D_1(x) &= \frac{1}{\sigma_X\sigma_Y(1-\rho^2)}\bigg(\frac{n}{4}\bigg)^{1-n/2}\bigg|\frac{\mu_X^2}{\sigma_X^2}-\frac{\mu_Y^2}{\sigma_Y^2}\bigg|^{1-n/2}\bigg(\frac{x}{\sigma_X\sigma_Y}\bigg)^{n/2} \\
					 &\quad \times \exp\bigg(-\frac{n}{2(1-\rho^2)}\bigg(\frac{\mu_X^2}{\sigma_X^2}+\frac{\mu_Y^2}{\sigma_Y^2}-\frac{2\rho\mu_X\mu_Y}{\sigma_X\sigma_Y}\bigg)
			      - \frac{x}{\sigma_X\sigma_Y(1+\rho)}\bigg)\,.
	 \end{align*}

\noindent 2. Suppose $\mu_X/\sigma_X=\mu_Y/\sigma_Y$. Then, for $x>0$,
   \begin{align}
	    f_{S_n}(x) &= D_2(x) \int_0^\infty (t^2(t+1))^{(n-2)/4} \exp\bigg(-\frac{2x t}{\sigma_X\sigma_Y(1-\rho^2)}\bigg)\nonumber\\
     &\quad\times I_{\frac{n}2-1}\bigg(\frac{2|\mu_X|}{\sigma_X}\frac{\sqrt{nx(1+t)}}{(1+\rho)\sqrt{\sigma_X\sigma_Y}}\bigg)	\,{\rm d} t,\label{222}
	 \end{align}	
where  
   \begin{align*}
	    D_2(x) &= \frac{n^{(2-n)/4}}{\sigma_X\sigma_Y(1-\rho)^{n/2}(1+\rho)\Gamma(n/2)}\bigg(\frac{|\mu_X|}{\sigma_X}\bigg)^{1-n/2}\bigg(\frac{x}{\sigma_X\sigma_Y}\bigg)^{(3n-2)/4} \\
     &\quad\times\exp\bigg(-\frac{n\mu_X^2}{\sigma_X^2(1+\rho)}
			      - \frac{x}{\sigma_X\sigma_Y(1+\rho)}\bigg)\,.
	 \end{align*}

\noindent 3. Suppose $\mu_X/\sigma_X=-\mu_Y/\sigma_Y$. Then, for $x>0$,
   \begin{align}
	    f_{S_n}(x) &= D_3(x) \int_0^\infty ((t+1)^2t)^{(n-2)/4} \exp\bigg(-\frac{2x t}{\sigma_X\sigma_Y(1-\rho^2)}\bigg) \nonumber\\
     &\quad\times I_{\frac{n}2-1}\bigg(\frac{2|\mu_X|}{\sigma_X}\frac{\sqrt{nxt}}{(1-\rho)\sqrt{\sigma_X\sigma_Y}}\bigg)	\,{\rm d} t,\label{333}
	 \end{align}	
where  
   \begin{align*}
	    D_3(x) &= \frac{n^{(2-n)/4}}{\sigma_X\sigma_Y(1+\rho)^{n/2}(1-\rho)\Gamma(n/2)}\bigg(\frac{|\mu_X|}{\sigma_X}\bigg)^{1-n/2}\bigg(\frac{x}{\sigma_X\sigma_Y}\bigg)^{(3n-2)/4} \\
					 &\quad \times \exp\bigg(-\frac{n\mu_X^2}{\sigma_X^2(1-\rho)}
			      - \frac{x}{\sigma_X\sigma_Y(1+\rho)}\bigg)\,.
	 \end{align*} 

\noindent 4. Formulas for the PDF $f_{S_n}(x)$ for $x<0$ follow from replacing $(x,\rho,\mu_Y)$ by $(-x,-\rho,-\mu_Y)$ in equations (\ref{Watson2}), (\ref{222}) and (\ref{333}).
\end{theorem}

\begin{remark}1. We observe that the integrand in the integral representation of the PDF of $S_n$ takes a simpler form in the cases $\mu_X/\sigma_X=\mu_Y/\sigma_Y$ and $\mu_X/\sigma_X=-\mu_Y/\sigma_Y$, with the integrand involving just a single modified Bessel function of the first kind in these cases, whilst the integrand involves a product of two modified Bessel functions of the first kind in the general case. 

\vspace{2mm}

\noindent 2. We can recover the formulas given in parts 2 and 3 of Theorem \ref{thmint} from the general formula (\ref{Watson2}) by taking the limits $\mu_X/\sigma_X-\mu_Y/\sigma_Y\rightarrow0$ and $\mu_X/\sigma_X+\mu_Y/\sigma_Y\rightarrow0$, respectively, using the limiting form (\ref{ilim}).

\vspace{2mm}

\noindent 3. The integrands in the integral representations given in Theorem \ref{thmint} take an elementary form if $n\geq1$ is an odd integer, as can be seen from the elementary representation (\ref{elem}) of the modified Bessel function of the first kind.
\end{remark}

\begin{proof} As usual, we prove the result for the case $\sigma_X=\sigma_Y=1$.
We will derive formula (\ref{Watson2}), which holds provided $|\mu_X|\not=|\mu_Y|$ (in the case $\sigma_X=\sigma_Y=1$). A slight modification of our derivation yields the formulas given in parts 2 and 3 of the theorem under the conditions $\mu_X=\mu_Y$ and $\mu_X=-\mu_Y$, respectively
We omit the details.

Suppose $|\mu_X|\not=|\mu_Y|$. We first prove the result for $x>0$.
We begin by re-writing \eqref{for1} into a more convenient form for our purposes, {\it viz.} 
   \begin{align} \label{for17}
      f_{S_n}(x) &= C(x) \sum_{k=0}^\infty \dfrac{\left( n/8\right)^k}{k!} \bigg(\frac{1-\rho}{1+\rho} 
						    (\mu_X +\mu_Y )^2 \bigg)^k \sum_{j=0}^k \frac{1}{\Gamma(n/2+k-j)}\binom{k}{j} 
								\bigg( \frac{1+\rho}{1-\rho} \frac{\mu_X -\mu_Y }{\mu_X+\mu_Y } \bigg)^{2j}\notag\\
             &\quad \times U\bigg(1-\frac{n}{2}-k+j,2-n-k,z\bigg),   
   \end{align}
where $z=2x/(1-\rho^2)$ and
   \[ C(x) = \frac{(1-\rho^2)^{n/2-1}}{2^{n-1}} \exp\bigg(-\frac{n(\mu_X^2+\mu_Y^2-2\rho\mu_X\mu_Y)}
			       {2(1-\rho^2)} - \frac{x}{1+\rho}\bigg)\,.\]
Applying Kummer's transformation (\ref{KummerU}) to the  confluent hypergeometric function of the second kind in (\ref{for17}) now yields          
   \begin{align*}
	    f_{S_n}(x) &= C(x) z^{n-1} \sum_{k=0}^\infty \dfrac{\left( n z/8 \right)^k}{k!} \bigg(\frac{1-\rho}{1+\rho} 
						    (\mu_X +\mu_Y )^2 \bigg)^k \notag \\
						 &\quad \times \sum_{j=0}^k \frac{1}{\Gamma(n/2+k-j)}\binom{k}{j} 
								\bigg( \frac{1+\rho}{1-\rho} \frac{\mu_X -\mu_Y }{\mu_X+\mu_Y } \bigg)^{2j}\, 
								U\bigg(\frac{n}2+j,n+k, z \bigg)\,.
	 \end{align*}
Applying the integral representation \eqref{LaplaceU}
followed by 
changing the order of summation and integration
results in 
   \begin{align*}
	    f_{S_n}(x) &= \frac{C(x) z^{n-1}}{\Gamma^2(n/2)} \int_0^\infty {\rm e}^{-z t} t^{n/2-1} (1+t)^{n/2-1} 
			          \sum_{k=0}^\infty \dfrac{[n z (1+t)]^k}{8^k\,k!} \bigg(\frac{1-\rho}{1+\rho} 
						    (\mu_X +\mu_Y )^2 \bigg)^k \notag \\
						 &\quad \times \sum_{j=0}^k \binom{k}{j} \frac{[t/(1+t)]^j}{(n/2)_{k-j} (n/2)_j} 
								\bigg(\frac{1+\rho}{1-\rho} \frac{\mu_X -\mu_Y}{\mu_X+\mu_Y} \bigg)^{2j}\, {\rm d} t\\ 
						 &= \frac{C(x) z^{n-1}}{\Gamma^2(n/2)} \int_0^\infty {\rm e}^{-z t} t^{n/2-1} (1+t)^{n/2-1} 
			          \sum_{k=0}^\infty \dfrac{[n z (1+t)]^k}{8^k\,(n/2)_k\,k!} \bigg(\frac{1-\rho}{1+\rho} 
						    (\mu_X +\mu_Y )^2 \bigg)^k \notag \\
						 &\quad \times {}_2F_1\bigg( - \frac{n}2-k+1, -k; \frac{n}2\ ; \frac{t}{1+t} 
								\bigg(\frac{1+\rho}{1-\rho}\bigg)^2 \bigg(\frac{\mu_X  
							- \mu_Y }{\mu_X+\mu_Y} \bigg)^2 \bigg)\,{\rm d} t,
	 \end{align*}
where the hypergeometric polynomial has degree ${\rm deg}({}_2F_1) = n/2+k-1$ for all fixed even positive integers $n \geq 2$, 
whilst ${\rm deg}({}_2F_1) = k$ for $n$ odd. Here, we applied the summation
   \[ \sum_{j=0}^k \binom{k}{j} \frac{w^j}{(u)_{k-j} (u)_j} = \frac1{(u)_k}\, {}_2F_1(-u-k+1, -k;\ u\ ; w)\,.\] 
Now, by virtue of Watson's formula \cite[p.\ 148, Eq.\ (2)]{GNW}, {\it viz.} 
   \begin{equation*} 
	    \sum_{k=0}^\infty \frac{(-1)^k y^{2k}}{(a)_k\,k!}\,{}_2F_1(-a-k+1, -k; b ; x^2) = \Gamma(a)\,\Gamma(b)\,x^{1-b}\, 
			           y^{2-a-b}\, J_{a-1}(2y)\, J_{b-1}(2xy)\,,
	 \end{equation*}
we conclude that
   \begin{align}
	    f_{S_n}(x) &= \frac{\tilde{D}_1(x)}{{\rm i}^{n-2}} \int_0^\infty (t(t+1))^{(n-2)/4} 
			          \exp\bigg(-\frac{2x t}{1-\rho^2}\bigg) \,
			          J_{\frac{n}{2}-1}\bigg(\mathrm{i}\sqrt{\frac{1+\rho}{1-\rho}}|\mu_X-\mu_Y|\sqrt{nzt/2}\bigg) \notag \\ 
						 &\quad \times J_{\frac{n}2-1}\bigg(\mathrm{i}\sqrt{\frac{1-\rho}{1+\rho}}|\mu_X+\mu_Y|\sqrt{nz(1+t)/2} \bigg)	\,{\rm d} t, \nonumber
	\end{align}	
where $\tilde{D}_1(x)$ is the coefficient $D_1(x)$ with $\sigma_X=\sigma_Y=1$, and we recall that $z=2x/(1-\rho^2)$. Finally, using that $J_\nu({\rm i}w) = {\rm i}^\nu\, I_\nu(w)$ (which is immediate from the definitions (\ref{jdef}) and (\ref{idef})) yields the formula (\ref{Watson2}) in the case $\sigma_X=\sigma_Y=1$, as required.

Finally, we consider the case $x<0$.
Observe that $Z=XY=_d-X'Y'$, where $(X', Y')$ is a bivariate normal random vector with mean vector $(\mu_X,-\mu_Y)$, variances $(\sigma_X^2,\sigma_Y^2)$ and correlation coefficient $-\rho$. Formulas for the PDF in the case $x<0$ thus follow from replacing $(x,\rho,\mu_Y)$ by $(-x,-\rho,-\mu_Y)$ in equations (\ref{Watson2}), (\ref{222}) and (\ref{333}).
\end{proof}


In the case that $\rho=0$ and one of the means is equal to zero, we can obtain an alternative integral representation for the density of $S_n$ from the Neumann series of the first type of the modified Bessel functions of the second kind 
 \eqref{simple2}.

\begin{theorem} \label{theorem7}
Let $\rho=0$, $\mu_Y=0$ and suppose $\mu_X \in \mathbb{R}$. Then, for $x \in \mathbb{R}$,
   \begin{align} \label{simple3}
      f_{S_n}(x) &= \frac{|x|^{n-1} }{\sqrt{\pi}\, 
						  ( 2\sigma_X \sigma_Y)^n \Gamma(n/2)}\exp\left(\!-\dfrac{n\mu_X^2}{2\sigma_X^2}\right)\nonumber\\
         &\quad\times\int_0^\infty t^{-(n+1)/2} \exp\bigg(\!-t -\frac{x^2}{4 \sigma_X^2 \sigma_Y^2 t}\bigg)\,
							 {}_0F_1\bigg(-;\frac{n}2; \frac{n \mu_X^2 x^2}{8 \sigma_X^4 \sigma_Y^2 t}\bigg)\, {\rm d}t. 
	 \end{align}
\end{theorem}

\begin{remark}
Bearing in mind relation (\ref{fi}), one can express the $_{0}F_1$ function in the integrand in formula (\ref{simple3}) in terms of the modified Bessel function of the first kind.    
\end{remark}

\begin{proof}
From the integral expression \eqref{watson1} we transform \eqref{simple2}, taking the shorthand $w := n \mu_X^2 |x|/4$,
getting 
   \begin{align} \label{for19}
	    f_{S_n}(x) &= \frac{|x|^{(n-1)/2}  }{2^{(n-1)/2} \sqrt{\pi}}\exp\bigg(\!-\frac{n\mu_X^2}{2}\bigg) \sum_{k=0}^{\infty}\frac{w^k}
								{k! \Gamma(n/2+k)} K_{\frac{n-1}2+k}(|x|) \notag \\
						 &= \frac{|x|^{n-1} }{2^n \sqrt{\pi} \Gamma(n/2)}\exp\bigg(\!-\frac{n\mu_X^2}{2}\bigg)\int_0^\infty t^{-(n+1)/2} \exp\bigg(\!-t -\frac{x^2}{4 t}\bigg)   
								\sum_{k=0}^\infty \frac{1}{(n/2)_k\, k!}\bigg(\frac{w |x|}{2t}\bigg)^k \, {\rm d}t. \nonumber
	 \end{align}						    
In turn, the inner-most series from above is detected as the confluent hypergeometric function
   \[ \sum_{k=0}^\infty \frac{1}{(n/2)_k \,k!}\left(\frac{w |x|}{2t}\right)^k = {}_0F_1\bigg(-;\frac{n}2; \frac{w |x|}{2t}\bigg).\] 
This completes the proving procedure.
\end{proof} 

\subsection{Infinite divisibility}

Our infinite divisibility result is given in the following theorem.

\begin{theorem}\label{thmif} 
For any $m\geq1$, we have the decomposition $Z=_d\sum_{i=1}^m U_{i,m}$, where $U_{1,m},\ldots, U_{m,m}$ are independent and identically distributed random variables with PDF
\begin{align}\label{pdfuu}
f_{U_{1,m}}(x)&=\frac{(1-\rho^2)^{q/2-1}}{2^{n-1}\sigma_X\sigma_Y}\exp\bigg(-\frac{q}{2(1-\rho^2)}\bigg(\frac{\mu_X^2}{\sigma_X^2}+\frac{\mu_Y^2}{\sigma_Y^2}-\frac{2\rho\mu_X\mu_Y}{\sigma_X\sigma_Y}\bigg)-\frac{|x|}{\sigma_X\sigma_Y(1+\rho\,\mathrm{sgn}(x))}\bigg)\nonumber\\
&\quad\times\sum_{k=0}^\infty\sum_{j=0}^k\frac{(q/8)^k}{k!\Gamma(q/2+a_{j,k}(x))}\binom{k}{j} \bigg(\frac{1+\rho}{1-\rho}\bigg)^j\bigg(\frac{\mu_X}{\sigma_X}-\frac{\mu_Y}{\sigma_Y}\bigg)^{2j} \bigg(\frac{1-\rho}{1+\rho}\bigg)^{k-j}\nonumber\\
&\quad\times \bigg(\frac{\mu_X}{\sigma_X}+\frac{\mu_Y}{\sigma_Y}\bigg)^{2k-2j} U\bigg(1-\frac{q}{2}-a_{j,k}(x),2-q-k,\frac{2|x|}{\sigma_X\sigma_Y(1-\rho^2)}\bigg),\quad x\in\mathbb{R},    
\end{align}
where $q=1/m$.
In particular, the distribution of $Z$ is infinitely divisible.
\end{theorem}

\begin{remark} We can extend the infinite divisibility result of Theorem \ref{thmif} to include the degenerate case $\rho\in\{-1,1\}$, so that distribution of the product $Z$ is infinitely divisible for $\mu_X,\mu_Y\in\mathbb{R}$, $\sigma_X,\sigma_Y>0$ and $-1\leq\rho\leq1$. This is because in the case $\rho\in\{-1,1\}$ the product $Z$ can be expressed in the form $Z=_d\rho T+a$, where $T$ is a non-central chi-square random variable and $a\in\mathbb{R}$ is a constant. Infinite divisibilty of the product $Z$ in the degenerate case $\rho\in\{-1,1\}$ is now immediate from the fact that the non-central chi-square distribution is infinitly divisible (see \cite{bose}). 
\end{remark}

\begin{proof}
We begin by confirming that the expression (\ref{pdfuu}) does indeed define the PDF of a real-valued random variable. This means that we need to show that, for any $m\geq1$, we have $f_{U_{1,m}}(x)\geq0$ for all $x\in\mathbb{R}$, and that $\int_{-\infty}^\infty f_{U_{1,m}}(x)\,\mathrm{d}x=1$. That $f_{U_{1,m}}(x)\geq0$ for all $x\in\mathbb{R}$ is immediate from the fact that all summands in the infinite series (\ref{pdfuu}) are non-negative for all $x\in\mathbb{R}$ (due to the positivity of the modified confluent hypergeometric function of the second kind; see (\ref{positive})). Also, following the steps of the proof of Theorem \ref{thm1} we see that, for $t\in\mathbb{R}$,
\begin{align}
\int_{-\infty}^\infty\mathrm{e}^{\mathrm{i}tx}f_{U_{1,m}}(x)\,\mathrm{d}x&=\frac1{([1-(1+\rho)\mathrm{i}t][1+(1-\rho)\mathrm{i}t])^{q/2}} \nonumber\\
&\quad\times
			             \exp\bigg(\frac{-q(\mu_X^2+\mu_Y^2-2\rho\mu_X\mu_Y)t^2+2q\mu_X\mu_Y\mathrm{i}t}
									 {2[1-(1+\rho)\mathrm{i}t][1+(1-\rho)\mathrm{i}t]}\bigg).  \label{one}   
\end{align}
Setting $t=0$ in (\ref{one}) yields $\int_{-\infty}^\infty f_{U_{1,m}}(x)\,\mathrm{d}x=1$, as required.

Since (\ref{pdfuu}) does indeed define the PDF of a real-valued random variable, we have that the characteristic function of $U_{1,m}$ is given by
\begin{align*}
\varphi_{U_{1,m}}(t)= \frac1{([1-(1+\rho)\mathrm{i}t][1+(1-\rho)\mathrm{i}t])^{q/2}} 
			             \exp\bigg(\frac{-q(\mu_X^2+\mu_Y^2-2\rho\mu_X\mu_Y)t^2+2q\mu_X\mu_Y\mathrm{i}t}
									 {2[1-(1+\rho)\mathrm{i}t][1+(1-\rho)\mathrm{i}t]}\bigg),
\end{align*} 
for $t\in\mathbb{R}$. Now, let $W_m=\sum_{i=1}^mU_{i,m}$. Therefore, since $U_{1,m},\ldots, U_{m,m}$ are independent and identically distributed random variables,
\begin{align*}
\varphi_{W_m}(t)=\prod_{i=1}^m\varphi_{U_{i,m}}(t)&=\frac1{([1-(1+\rho)\mathrm{i}t][1+(1-\rho)\mathrm{i}t])^{1/2}} \\
&\quad \times
			             \exp\bigg(\frac{-(\mu_X^2+\mu_Y^2-2\rho\mu_X\mu_Y)t^2+2\mu_X\mu_Y\mathrm{i}t}
									 {2[1-(1+\rho)\mathrm{i}t][1+(1-\rho)\mathrm{i}t]}\bigg), \quad t\in\mathbb{R},
\end{align*}
which we recognise as the characteristic function of the product $Z$. Hence, by the uniqueness of characteristic functions, $W_m=_d Z$, for any $m\geq1$. This completes the proof.
\end{proof}

\appendix

\section{Special functions}\label{appa}
In this appendix, we define the special functions that are used in this paper and state some of their relevant basic properties. Unless otherwise stated, these  properties can be found in the standard reference \cite{olver}. 

The \emph{Bessel function of the first kind} is defined, for $\nu\in\mathbb{R}$ and $z\in\mathbb{C}$, by the power series
\begin{equation}\label{jdef}
J_\nu(z)=\sum_{k=0}^\infty (-1)^k\frac{(z/2)^{2k+\nu}}{k!\Gamma(k+\nu+1)}.    
\end{equation}

The \emph{modified Bessel function of the first kind} is defined, for $\nu\in\mathbb{R}$ and $z\in\mathbb{C}$, by the power series
\begin{equation}\label{idef}
I_\nu(z)=\sum_{k=0}^\infty \frac{(z/2)^{2k+\nu}}{k!\Gamma(k+\nu+1)}.    
\end{equation}
For $\nu=n-1/2$, $n=0,1,2,\ldots$, the modified Bessel function of the first kind takes an elementary form: for $x\in\mathbb{R}$,
\begin{align}\label{elem} I_{-\frac{1}{2}}(x)&=\sqrt{\frac{2}{\pi x}}\cosh(x), \\
I_{n+\frac{1}{2}}(x)&=\frac{1}{\sqrt{2\pi x}}\bigg\{\sum_{j=0}^n\frac{(-1)^j(n+j)!}{j!(n-j)!}\frac{\mathrm{e}^x}{(2x)^j}+(-1)^{n+1}\sum_{j=0}^n\frac{(n+j)!}{j!(n-j)!}\frac{\mathrm{e}^{-x}}{(2x)^j}\bigg\}.
\end{align}
The function 
$I_\nu(x)$ has the following limiting behaviour as $x\rightarrow0$,
\begin{align}
I_\nu(x)&\sim\frac{x^\nu}{2^{\nu}\Gamma(\nu+1)}, \quad \nu>-1.\label{ilim}
\end{align}

The \emph{modified Bessel function of the second kind}
can be defined, for $\nu\in\mathbb{R}$, by the Laplace--transform type integral \cite[p. 183, Eq.(15)]{GNW}\footnote{By Watson, this integral representation of 
the modified Bessel function of the second kind was probably firstly considered by Poisson in an equivalent form in 
\cite[p. 237]{Poisson}.}
   \begin{equation} \label{watson1}
	    K_\nu(x) = \frac12 \Big(\frac{x}2\Big)^\nu \int_0^\infty \exp\bigg(-t-\frac{x^2}{4 t}\bigg)\, 
			           \frac{{\rm d} t}{t^{\nu+1}}\,, \quad x>0\,.
	 \end{equation}
It is clear from (\ref{watson1}) that, for $\nu\in\mathbb{R}$,
\begin{equation}
\label{knon} K_{\nu}(x)>0,\quad x>0.    
\end{equation}
We have the following identity: for $\nu\in\mathbb{R}$ and $x>0$,
\begin{equation}
\label{par} K_{-\nu}(x)=K_{\nu}(x).    
\end{equation}

The \emph{generalized hypergeometric function} is defined, for $|x|<1$, by the power series
\begin{equation}
\label{gauss}
{}_pF_q(a_1,\ldots,a_p; b_1,\ldots,b_q;x)=\sum_{j=0}^\infty\frac{(a_1)_j\cdots(a_p)_j}{(b_1)_j\cdots(b_q)_j}\frac{x^j}{j!},
\end{equation}
and by analytic continuation elsewhere. The ascending factorial is given by $(u)_j=u(u+1)\cdots(u+j-1)$. 
   
The {\it Tricomi - or second kind - confluent hypergeometric function} can be defined by
\begin{equation*}
U(a,b,x)=\frac{\Gamma(b-1)}{\Gamma(a)}x^{1-b}{}_{1}F_1(a-b+1,2-b,x)+\frac{\Gamma(1-b)}{\Gamma(a-b+1)}{}_{1}F_1(a,b,x), \quad b\notin\mathbb{Z},   
\end{equation*}
and
\begin{equation*}
U(a,b,x)=\lim_{\beta\rightarrow b}U(a,\beta,x),\quad b\in\mathbb{Z}.   
\end{equation*}
If $a+1\geq b$, then
\begin{equation}\label{positive}
U(a,b,x)>0,\quad x>0.    
\end{equation}
We also make use of the Kummer's transformation 
   \begin{equation} \label{KummerU}
	    U(a, b, x) = x^{1-b}\, U(a-b+1, 2-b, x),
	 \end{equation}
and the Laplace transform integral representation formula,
 which reads
   \begin{equation} \label{LaplaceU}
	    U(a, b, x) = \frac1{\Gamma(a)} \int_0^\infty {\rm e}^{-x t}\, t^{a-1} (1+t)^{b-a-1}\,{\rm d}t, 
			             \quad a>0;\, x>0.
	 \end{equation}
We have the relations
\begin{align}
\label{fi}{}_0F_1(-;b; x) &= \Gamma(b)\, x^{(1-b)/2}\, I_{b-1}(2 \sqrt{x}),\\
\label{uk} U(a,2a,2x)&=\frac{1}{\sqrt{\pi}}\mathrm{e}^{x}(2x)^{1/2-a}K_{a-\frac{1}{2}}(x).
\end{align}
The function $U(a,b,x)$ has the following limiting behaviour as $x\rightarrow0$,
\begin{align}
U(a,1,x)&= -\frac{\log|x|}{\Gamma(a)}+O(1),\label{00}\\ 
U(a,b,x)&= \frac{\Gamma(1-b)}{\Gamma(a-b+1)}+o(1), \quad b<1,\;a-b+1>0.\label{000}
\end{align}


The \emph{Whittaker function} is defined, for $x\in\mathbb{R}$, by
\begin{equation}\label{link}
W_{\kappa,\mu}(x)=\mathrm{e}^{-x/2}x^{\mu+1/2}U\bigg(\frac{1}{2}+\mu-\kappa,1+2\mu,x\bigg).  
\end{equation}

The following formula corrects the definite integral formula given by \cite[p.\ 325, Eq.\ 19]{integralbook}:
\begin{equation}\label{int1}
\int_{-\infty}^\infty\mathrm{e}^{\mathrm{i}x t}(z+\mathrm{i}t)^{-\rho}(y-\mathrm{i}t)^{-\sigma}\,\mathrm{d}t=\frac{2\pi}{\Gamma(\delta)}(y+z)^{1-\rho-\sigma}\mathrm{e}^{-|x|\theta}U(1-\delta,2-\rho-\sigma,|x|(y+z)),
\end{equation}
where $y>0$, $z>0$; $\rho+\sigma>1$; $\delta=\sigma$, $\theta=y$ for $x<0$, and $\delta=\rho$, $\theta=z$ for $x\geq0$. The formula (\ref{int1}) is also valid in the case $\rho=\sigma=1/2$, which can be seen from Lemma 3.1 of \cite{np16} and an application of the formula (\ref{uk}). The correction we make is that the formula of \cite{integralbook} contains an additional factor of $x^2$, which should not be present. The formula (\ref{int1}) can alternatively be derived from equation 3.384(9) of \cite{g07} together with an application of formula (\ref{link}); the formula of \cite{g07} is expressed in terms of the Whittaker function rather than the confluent hypergeometric function of the second kind.
The following special case of the integral formula (\ref{int1}) is given by equation 3.384(9) of \cite{g07}:
\begin{align}
\int_{-\infty}^\infty \frac{\mathrm{e}^{\mathrm{i}x t}}{(1+t^2)^\rho}\,\mathrm{d}t&=\frac{2^{1-\rho}\pi}{\Gamma(\rho)}|x|^{\rho-1}W_{0,\frac{1}{2}-\rho}(2|x|)\nonumber\\
&=\frac{2^{3/2-\rho}\sqrt{\pi}}{\Gamma(\rho)}|x|^{\rho-1/2}K_{\rho-\frac{1}{2}}(|x|),\label{int11}
\end{align}
where we obtained the second equality using the standard relation $W_{0,\nu}(2x)=\sqrt{2x/\pi}K_\nu(x)$ and equation (\ref{par}). 

\section*{Acknowledgements} 

RG is funded in part by EPSRC grant EP/Y008650/1 and EPSRC grant UKRI068. TP was partially supported by the University of Rijeka under the project {\tt uniri-iskusni-prirod-23-98}. We would like to thank the reviewer for carefully reading our paper and for their helpful comments and suggestions.

\end{document}